\documentclass[11pt,oneside,english]{amsart}
\usepackage[T1]{fontenc}
\usepackage[latin9]{inputenc}
\usepackage{amstext}
\usepackage{amsthm}
\usepackage{amssymb}

\makeatletter
\numberwithin{equation}{section}
\numberwithin{figure}{section}
\theoremstyle{plain}
\newtheorem{thm}{\protect\theoremname}
\theoremstyle{remark}
\newtheorem{rem}[thm]{\protect\remarkname}
\theoremstyle{plain}
\newtheorem{prop}[thm]{\protect\propositionname}
\theoremstyle{plain}
\newtheorem{lem}[thm]{\protect\lemmaname}
\theoremstyle{plain}
\newtheorem{cor}[thm]{\protect\corollaryname}
\theoremstyle{remark}
\newtheorem{claim}[thm]{\protect\claimname}

\usepackage{hyperref}

\makeatother

\usepackage{babel}
\providecommand{\claimname}{Claim}
\providecommand{\corollaryname}{Corollary}
\providecommand{\lemmaname}{Lemma}
\providecommand{\propositionname}{Proposition}
\providecommand{\remarkname}{Remark}
\providecommand{\theoremname}{Theorem}

\usepackage{csquotes}

\usepackage[
backend=biber,
style=alphabetic,
sorting=ynt
]{biblatex}

\title{Bibliography management: \texttt{biblatex} package}
\date{ }

\begin{document}
\title{Schwartz Functions, Hadamard products, and the Dixmier-Malliavin theorem}
\author{Devadatta G. Hegde}
\date{May 4, 2021}
\email{hegde039@umn.edu}
\curraddr{Department of Mathematics, University of Minnesota, Minneapolis, MN
- 55414. }
\begin{abstract}
In this paper we show that functions of the form $\prod_{n\ge1}\frac{1}{\left(1+\frac{x^{2}}{a_{n}^{2}}\right)}$
where $a_{n}>0$ and $\sum_{n\ge1}\frac{1}{a_{n}^{2}}<\infty$ are
in the Schwartz space of the real line, answering a question raised
by Casselman in \cite{CasselmanDM}. As a consequence we obtain substantial
simplifications in the proofs of Dixmier and Malliavin of their theorem
that every test function on a Lie group is a finite linear combination
of convolutions of two test functions, and an analogue of this for
Fr\'echet space Lie group representations. 
\end{abstract}

\maketitle
\tableofcontents{}

\section{Introduction }

In the context of trace formulas, for example the Selberg trace formula
for compact quotients, the result of Dixmier and Malliavin below
can play an important role, as follows. Consider the right regular
action $\pi$ of $G$ on $L^{2}(\Gamma\backslash G)$ where $\Gamma$
is a discrete subgroup of a semi-simple real Lie group $G$ and the
quotient $\Gamma\backslash G$ is compact. The trace formula expresses
the trace of the operator $\pi(\phi)$ for a test function $\phi$
in two different ways. It essential to know that the operator $\pi(\phi)$
is of trace class before taking the trace (see \cite{Arthur-trace},
section 1). 

In the above case the action of $\phi$ on $L^{2}(\Gamma\backslash G)$
is given by a continuous kernel $K(x,y)=\sum_{\gamma\in\Gamma}\phi(x^{-1}\gamma y)$
that is square integrable i.e, $K\in L^{2}(\Gamma\backslash G\times\Gamma\backslash G$).
Square integrability follows from the compactness of $\Gamma\backslash G$.
Therefore the operator $\pi(\phi)$ is Hilbert-Schmidt. An operator on
a Hilbert space is trace class if it is a finite linear combination
of compositions of pairs of Hilbert-Schmidt operators. One way to
prove that $\pi(\phi)$ is trace class is by using the theorem of
Dixmier and Malliavin that that every test function on a Lie group
is a finite linear combination of convolutions of two test functions and that the convolution
$f*g$ of two test functions on $G$ is a test function which acts
by $\pi(f)\circ\pi(g)$ on the representation space.
This was proved in the seminal paper of Dixmier and Malliavin \cite{Dixmier-Malliavin}.
Their result is: 
\begin{thm}
\label{thm:intro DixMal}(a) Every test function $\phi$ on a real
Lie group is a finite sum of convolutions $\phi_{i}*\psi_{i}$ where
$\phi_{i},\psi_{i}$ are test functions. 

(b) Let $(\pi,V)$ be a continuous Fr\'echet space representation of
a real Lie group $G$ and let $V^{\infty}$ be the Fr\'echet space of
smooth vectors in $V$. Then every smooth vector can be written as
a finite linear combination of vectors of the form $\pi(f)v$ where
$f$ is a test function and $v\in V^{\infty}$. 
\end{thm}

In the compact quotient case the action of $\pi(f)$ is Hilbert-Schmidt
for $f$ merely continuous. Thus a result of
Cartier in \cite{Cartier-testVector}, which shows that any test function
can be written as a finite linear sum of $m$-times differentiable
functions for all $m\ge0$, is sufficient to show that $\pi(f)$ is
trace class. But the result of Dixmier and Malliavin, in this generality,
is an essential technical result for many problems in the archimedean
local aspects in the theory of automorphic forms. See \cite{Cogdell_RankinSelberg}
for an application in the determination of poles of Rankin-Selberg
L-functions for $GL_{n}$. 

The technical heart of Dixmier and Malliavin's proof is that entire
functions of the form $f(x)=\prod_{n\ge1}\left(1+\frac{x^{2}}{a_{n}^{2}}\right)$
for $a_{n}>0$ and $\frac{a_{n+1}}{a_{n}}\ge2$ for all $n\ge1$ have
the property that $\frac{1}{f(x)}\in\mathcal{S}$, the Schwartz space
of the real line. This allows them to construct a sufficiently large
family of sequences $\{c_{n}\}_{n\ge1}$ such that for each such sequence
there exists $\phi\in\mathcal{S}$ satisfying 
\[
\sum_{j=1}^{n}c_{j}\frac{d^{2j}\phi}{dx^{2j}}\to\delta\hspace{1em}\text{as }n\to\infty
\]
 in $\mathcal{S}'$, the dual of $\mathcal{S}$. From this a strong
factorization theorem that every function in $\mathcal{S}$ is a convolution
$f*g$ for $f,g\in\mathcal{S}$ can be easily deduced (see theorem
\ref{thm: Schwartz}). This result is the key for all further generalizations. 

In his exposition of the Dixmier-Malliavin theorem, Casselman \cite{CasselmanDM}
noted that the function
\[
\frac{x}{\sinh x}=\frac{1}{\prod_{n\ge1}\left(1+\frac{x^{2}}{\pi^{2}n^{2}}\right)}
\]
is in the Schwartz space and hence entire functions of the form $f(x)=\prod_{n\ge1}\left(1+\frac{x^{2}}{a_{n}^{2}}\right)$
might have the property that $\frac{1}{f}\in\mathcal{S}$ under very
general circumstances and asked for a good criterion on sequences
$a_{n}$ so that $\frac{1}{f}\in\mathcal{S}$. 

The main result (theorem \ref{Main theorem}) of this paper is a positive
answer to the above question of Casselman showing that if $f(x)=\prod_{n\ge1}\left(1+\frac{x^{2}}{a_{n}^{2}}\right)$
converges to an entire function then $\frac{1}{f}\in\mathcal{S}$,
by an application of Fourier transform methods. This simplifies the
technical aspects of the Dixmier-Malliavin's proof considerably. For
the convenience of the reader we give an amplified version of the
Dixmier-Malliavin proof in section \ref{sec:Dixmier-Malliavin-theorem}
incorporating the simplifications due to theorem \ref{Main theorem}
of this paper. The proof itself has a natural interpretation in terms
of Bornological spaces and may enable a conceptual proof that Godement-Jacquet
zeta integrals give the standard $L$-function attached to a cuspform
on $GL_{n}$ (see \cite{GalDordixmiermalliavin}).

\textbf{Acknowledgements.} I would like to thank professor Bill Casselman
and professor Paul Garrett for their careful reading of the manuscript
and helpful remarks. I would like to thank professor 
Abdelmalek Abdesselam for bringing Miyazaki's work \cite{Miyazaki} to my attention.

// Comments from the readers are welcome. 

\section{The Main theorem}

In the following we use the standard notation: $z=x+iy$ is a complex
number, $\mathcal{S=S}(\mathbb{R})$ denotes the Schwartz space of
rapidly decreasing smooth functions on the real line $\mathbb{R}$
whose derivatives are also rapidly decreasing, and $\widehat{f}$
denotes the Fourier transform of a function or a tempered distribution
$f$. 

The main result of this paper is:
\begin{thm}
\label{Main theorem}Let $\{a_{n}\}_{n\ge1}$ be a sequence of positive
real numbers such that the following conditions are true.

(a) The sequence $a_{n}$ monotonically increases to infinity. That
is, $a_{n+1}\ge a_{n}$ for all $n\ge0$ and $a_{n}\to\infty$.

(b) The infinite product $f(z)=\prod\left(1+\frac{z^{2}}{a_{n}^{2}}\right)$
converges uniformly on compact subsets of $\mathbb{C}$. 

Then, $f|_{\mathbb{R}}=f(x)$ is in the Schwartz space of the real
line. 
\end{thm}

\begin{rem}
Any sequence $a_{n}\ge n$ would satisfy the hypothesis of the theorem,
but we do not assume this. 
\end{rem}

For the proof we need a few elementary results from distribution theory
and Fourier transform. The following Fourier transform is well known
(Poisson kernel) and can be computed by contour integration: for $\xi\in\mathbb{R}$
and $a>0$
\begin{equation}
\int_{\mathbb{R}}\frac{e^{-2\pi i\xi x}}{x^{2}+a^{2}}dx=\frac{\pi}{a}e^{-2\pi a|\xi|}
\end{equation}

The following proposition about the distributional derivatives of
non-smooth functions is standard (see theorem 3.1.3 in H\"ormander \cite{Hormander}).
\begin{prop}
\label{prop:distibutional-derivative}Let $f$ be a continuously differentiable
function on $\mathbb{R}\backslash\{0\}$ such that $f_{0^{+}}=\lim_{x\to0^{+}}f(x)$
and $f_{0^{-}}=\lim_{x\to0^{-}}f(x)$ exist. Let $v(x)=\frac{df}{dx}$
for $x\neq0$. Then the distributional derivative is $f'=v+(f_{0^{+}}-f_{0^{-}})\delta$
where $\delta$ is the Dirac delta distribution at the origin. In
particular, if $f$ is continuous at $0$ then $f'=v$.
\end{prop}

\begin{lem}
\label{lem:differentiation}Let $f(\xi)$ be a rapidly decreasing
smooth function. The convolution $e^{-2\pi a|\xi|}*f$ is a smooth
function and its derivative is $\frac{d}{dx}(e^{-2\pi a|\xi|}*f)(x)=(\psi*f)(x)$,
where
\[
\psi(\xi)=-\text{sgn}(\xi)2\pi ae^{-2\pi a|\xi|}=\begin{cases}
-2\pi ae^{-2\pi a\xi} & \text{for }\xi>0\\
2\pi ae^{2\pi a\xi} & \text{for }\xi<0
\end{cases}
\]
 is the distributional derivative of $e^{-2\pi a|\xi|}$. 
\end{lem}

\begin{proof}
This is an elementary computation. We have 
\[
(e^{-2\pi a|\xi|}*f)(x)=\int_{\mathbb{R}}e^{-2\pi a|x-\xi|}f(\xi)d\xi=\int_{-\infty}^{x}e^{2\pi a(\xi-x)}f(\xi)d\xi+\int_{x}^{\infty}e^{2\pi a(x-\xi)}f(\xi)d\xi
\]
\[
=e^{-2\pi ax}\int_{-\infty}^{x}e^{2\pi a\xi}f(\xi)d\xi+e^{2\pi ax}\int_{x}^{\infty}e^{-2\pi a\xi}f(\xi)d\xi
\]
 so $e^{-2\pi a|\xi|}*f$ is a smooth function. Therefore,
\[
\frac{d}{dx}(e^{-2\pi a|\xi|}*f)(x)=\left(f(x)-2\pi ae^{-2\pi ax}\int_{-\infty}^{x}e^{2\pi a\xi}f(\xi)d\xi\right)
\]
 
\[
+\left(-f(x)+2\pi ae^{2\pi ax}\int_{x}^{\infty}e^{-2\pi a\xi}f(\xi)d\xi\right)
\]
 
\[
=\int_{-\infty}^{x}(-2\pi a)e^{2\pi a(\xi-x)}f(\xi)d\xi+\int_{x}^{\infty}(2\pi a)e^{2\pi a(x-\xi)}f(\xi)d\xi=(\psi*f)(x)
\]
That the distributional derivative of $e^{-2\pi a|\xi|}$ is $\psi$
follows from proposition \ref{prop:distibutional-derivative}. 
\end{proof}
\begin{lem}
\label{lem:decay}Let $f$ and $g$ be rapidly decreasing functions
on $\mathbb{R}$ which are continuous except possibly at $0$ where
there is at most a jump discontinuity. Then the convolution $f*g$
is also a rapidly decreasing function. 
\end{lem}

\begin{proof}
We have 
\[
|f*g(x)|\le\int_{\mathbb{R}}|f(x-y)g(y)|dy
\]
\[
\le\int_{|x-y|<\frac{|x|}{2}}|f(x-y)g(y)|dy+\int_{|x-y|\ge\frac{|x|}{2}}|f(x-y)g(y)|dy
\]
 
\[
\le\left(\sup_{y\in\mathbb{R}}|f(y)|\right)\int_{|y|>\frac{|x|}{2}}|g(y)|dy+\left(\sup_{y\in\mathbb{R}}|g(y)|\right)\int_{|x-y|\ge\frac{|x|}{2}}|f(x-y)|dy
\]
 and the result follows since $\int_{|y|>\frac{|x|}{2}}|g(y)|dy$
and $\int_{|x-y|\ge\frac{|x|}{2}}|f(x-y)|dy$ are rapidly decreasing
at infinity. 
\end{proof}
\begin{proof}
(of the theorem) The function $\phi(x)=\prod_{n\ge1}\left(1+\frac{x^{2}}{a_{n}^{2}}\right)$
grows faster than any polynomial and hence $f(x)=\frac{1}{\phi(x)}$
is rapidly decreasing. We estimate
\[
\left|1+\frac{z^{2}}{a^{2}}\right|=\left|1+\frac{(x+iy)^{2}}{a^{2}}\right|=\frac{1}{a^{2}}\left|(x^{2}+a^{2}-y^{2})+i2xy\right|\ge\left|1+\frac{x^{2}-y^{2}}{a^{2}}\right|
\]
 Thus, 
\[
|f(x+iy)|=\left|\frac{1}{\prod_{n\ge1}\left(1+\frac{(x+iy)^{2}}{a_{n}^{2}}\right)}\right|\le\frac{1}{\prod_{n\ge1}\left|1+\frac{(x^{2}-y^{2})}{a_{n}^{2}}\right|}
\]
The above estimate shows that $f(x+iy)$ is rapidly decreasing for
any fixed $y$ and we move the contour of integration to a line $y=\epsilon$
where there are no poles of $f$ in the strip $0<y\le\epsilon$. Thus,
\[
\widehat{f}(\xi)=\int_{\mathbb{R}}f(x)e^{-2\pi i\xi x}dx=e^{2\pi\epsilon\xi}\int_{\mathbb{R}}f(x+i\epsilon)e^{-2\pi i\xi x}dx
\]
and

\[
|\widehat{f}(\xi)|\le Ce^{2\pi\epsilon\xi}\hspace*{1em}\text{( for some }C>0\text{)}
\]
 so $\widehat{f}$ is rapidly decreasing at $-\infty$. Choosing an
$\epsilon<0$, we get that $\widehat{f}$ is rapidly decreasing at
$+\infty$. 

To prove that $f\in\mathcal{S}$ we show that $\widehat{f}\in\mathcal{S}$
by proving that all the derivatives of $\widehat{f}$ are also rapidly
decreasing. Since $f$ is rapidly decreasing $\widehat{f}$ is a smooth
function. We write $a=a_{1}$ and
\[
f=\frac{1}{\left(1+\frac{x^{2}}{a^{2}}\right)}\cdot f_{1}\hspace{1em}\text{where }f_{1}=\frac{1}{\prod_{n\ge2}\left(1+\frac{x^{2}}{a_{n}^{2}}\right)}.
\]
Using the Fourier transform of $\frac{1}{x^{2}+a^{2}}$ we get
\[
\widehat{f}(\xi)=\widehat{\left(\frac{a^{2}}{x^{2}+a^{2}}\right)}*\widehat{f}_{1}=\pi a\left(e^{-2\pi a|\xi|}\right)*\widehat{f_{1}}
\]
By using lemma \ref{lem:differentiation}, we get

\[
\frac{d\widehat{f}}{d\xi}=\pi a\left(\psi*\widehat{f_{1}}\right),\hspace{1em}\psi(\xi)=-\text{sgn}(\xi)2\pi ae^{-2\pi a|\xi|}
\]
 and from lemma \ref{lem:decay} that it is a rapidly decreasing function . It follows
that $\frac{d\widehat{f_{1}}}{dx}$ is also rapidly decreasing since
it is an infinite product of the same form as $f$. Further,
\[
\frac{d^{n}\hat{f}}{d\xi^{n}}=\pi a\left(\psi*\frac{d^{n-1}\hat{f_{1}}}{d\xi^{n-1}}\right)
\]
So by induction on $n$, using lemma \ref{lem:decay}, we conclude
that $\frac{d^{n}\hat{f_{1}}}{d\xi^{n}}$ is rapidly decreasing for
any $n$. 
\end{proof}

\section{\label{sec:Dixmier-Malliavin-theorem}The Dixmier-Malliavin Theorem}

In this section we amplify Dixmier and Malliavin's proof of their
theorems mentioned in the introduction simplifying some technical
lemmas by using theorem \ref{Main theorem}. 

First we need a technical result (corollary \ref{cor:bounds b_n Schwartz})
about Schwartz functions on the real line, using which we illustrate
the proof mechanism of the general case by proving
the strong factorization property of Schwartz functions on the real line (see the proof of theorem \ref{thm: Schwartz}). 

In subsection \ref{subsec:Some-technical-results.} we use this technical
result about Schwartz functions to deduce a similar result about test
functions (lemma \ref{lem: bounds test}) which gives a \emph{weak
}factorization. Then we collect an amusing trick (lemma \ref{lem:convolution of distribution}),
due to Dixmier and Malliavin, about convolution of compactly supported
measures which reduces the Lie group case to the case of test functions
on the real line. 

\subsection{A special case of the Dixmier-Malliavin theorem: $\mathcal{S}(\mathbb{R})$ }

In this subsection we prove
\begin{thm}
\label{thm: Schwartz}Every function $\phi\in\mathcal{S}$ has a factorization
$\phi=f*g$ where $f,g\in\mathcal{S}$.
\end{thm}

This result was originally proved for $\mathbb{R}^{n}$ by Miyazaki \cite{Miyazaki} using different methods. As in Dixmier-Malliavin \cite{Dixmier-Malliavin} we need a sufficiently
large family of sequences $\{c_{n}\}_{n\ge1}$ such that for each
such sequence there exists a $\phi\in\mathcal{S}$ such that $\sum_{n\ge j\ge1}c_{j}\frac{d^{2j}\phi}{dx^{2j}}\to\delta$
in $\mathcal{S}'$. By using the Fourier transform this can be reduced
to a problem about entire functions answered by using the main theorem
\ref{Main theorem} and the Hadamard factorization theorem (see Chapter
5 of \cite{Stein-Shakarchi}, for example). We begin by collecting
a few elementary results about entire functions. 
\begin{lem}
The function $f(z)=\sum_{n\ge0}\frac{z^{n}}{(n!)^{4}}$ satisfies
$|f(z)|\le Ce^{c|z|^{\frac{1}{4}}}$ for some real numbers $c$ and $C$. 
\end{lem}

\begin{proof}
Consider the function $g(z)=\sum_{n\ge0}\frac{z^{n}}{(n!)^{2}}$.
Since 
\[
\sum_{n\ge0}a_{n}b_{n}\le\left(\sum_{n\ge0}a_{n}\right)\left(\sum_{n\ge0}b_{n}\right)
\]
 for $a_{n},b_{n}>0$, we have 
\[
|g(z)|\le\sum_{n\ge0}\frac{|z|^{n}}{(n!)^{2}}\le\left(\sum_{n\ge0}\frac{\left(|z|^{\frac{1}{2}}\right)^{n}}{n!}\right)\left(\sum_{n\ge0}\frac{\left(|z|^{\frac{1}{2}}\right)^{n}}{n!}\right)\le e^{2|z|^{\frac{1}{2}}}
\]
 Similarly, 
\[
|f(z)|\le|f(|z|)|\le g(|z|^{\frac{1}{2}})\cdot g(|z|^{\frac{1}{2}})\le e^{4|z|^{\frac{1}{4}}}
\]
\end{proof}
\begin{cor}
\label{cor:entire-func}If $\frac{1}{\left(n!\right)^{4}}\ge c_{n}>0$
then the function $f(z)=\sum_{n\ge0}c_{n}z^{2n}$ is an infinite
product 
\[
f(z)=\prod_{n\ge1}\left(1+\frac{z^{2}}{a_{n}^{2}}\right)
\]
 for some $a_{n}>0$. 
\end{cor}

\begin{proof}
Note that $f$ is an entire function, and all the zeros of $f$
are off the real line and occur in conjugate pairs. The existence
of such an infinite product follows from Hadamard factorization since
$|f(z)|\le e^{4|z|^{\frac{1}{2}}}$ has order of growth $\le\frac{1}{2}$. 
\end{proof}
The following is the crucial property of functions in $\mathcal{S}$
used in the proof of the Dixmier-Malliavin theorem. 
\begin{cor}
\label{cor:bounds b_n Schwartz}(Dixmier-Malliavin) Given a sequence
of bounds $B_{n}>0$, there exists a sequence of constants $0<b_{n}<B_{n}$
and a function $\psi\in\mathcal{S}$ such that 
\[
F_{\psi,n}=\sum_{j=0}^{n}(-1)^{j}b_{j}\psi^{(2j)}\to\delta\hspace*{1em}(\text{in }\mathcal{S}')
\]
as  $n\to\infty$
\end{cor}

\begin{proof}
Since Fourier transform is a topological isomorphism on $\mathcal{S}'$
it is enough to show that there exists a $\psi\in\mathcal{S}$ such
that

\[
\widehat{F_{\psi,n}}(\xi)=\left(\sum_{0\le j\le n}(2\pi)^{j}b_{j}\xi^{2j}\right)\widehat{\psi}(\xi)\to1
\]
in $\mathcal{S}'$. By theorem \ref{Main theorem} we pick $\psi\in\mathcal{S}$
such that 
\[
\widehat{\psi}(\xi)=\frac{1}{\sum_{j\ge0}(2\pi)^{j}b_{j}\xi^{2j}}=\frac{1}{\prod_{n\ge1}\left(1+\frac{\xi^{2}}{a_{n}^{2}}\right)}
\]
 with $b_{j}<\frac{1}{(2\pi)^{j}(j!)^{4}}$.
The result now follows since 
\[
\frac{\sum_{0\le j\le n}(2\pi)^{j}b_{j}\xi^{2j}}{\sum_{j=0}^{\infty}(2\pi)^{j}b_{j}\xi^{2j}}\to1
\]
 uniformly on compact subsets of $\mathbb{R}$ as $n\to\infty$ and
hence also as tempered distributions. 
\end{proof}
By using the corollary \ref{cor:bounds b_n Schwartz} with bounds
$B_{j}$ above, we choose a $\psi\in\mathcal{S}$ such that 
\[
\sum_{0\le j\le n}(-1)^{j}b_{j}\psi^{(2j)}\to\delta\text{ \ \ \ (in \ensuremath{\mathcal{S}'}, with \ensuremath{0<b_{j}<B_{j})}}
\]
On one hand,
\[
\phi*\left(\sum_{j\le n}(-1)^{j}b_{j}\psi^{(2j)}\right)\to\phi\hspace*{1em}\text{ (in }\mathcal{S}'\text{)}
\]
On the other hand, 
\[
\phi*\left(\sum_{j\le n}(-1)^{j}b_{j}\psi^{(2j)}\right)=\psi*\left(\sum_{j\le n}(-1)^{j}b_{j}\phi^{(2j)}\right)
\]
 We choose bounds $B_{j}$ to make $\sum_{j\le n}(-1)^{j}b_{j}\phi^{(2j)}$
converge in $\mathcal{S}$. Schwartz space $\mathcal{S}$ a Fr\'echet
space given by semi-norms 
\[
|f|_{m,n}=\sup_{x\in\mathbb{R}}|x^{m}f^{(n)}(x)|
\]
 To show that a sequence $\varphi_{n}\in\mathcal{S}$ converges in
$\mathcal{S}$ it is enough to show that the sequence is Cauchy for
the each of the norms $|\cdot|_{m,n}$. That is, 
\[
\sup_{x\in\mathbb{R}}|x^{m}(\varphi_{i}^{(n)}-\varphi_{j}^{(n)})|\to0\hspace{1em}\text{as }i,j\to\infty
\]
This is follows if the bounds $B_{j}>0$ satisfy
\[
\sum_{j\ge0}B_{j}.\sup_{x\in\mathbb{R}}|x|^{m}|\phi^{(2j+n)}|<\infty\text{ \ \ \  for all \ensuremath{n,m}}
\]
 which we obtain by the following lemma \ref{lem:diagonal} .
\begin{lem}
\label{lem:diagonal}Given a sequence of positive numbers $b_{mn}$
$(0\le m,n\in\mathbb{Z})$ we can find a sequence $a_{j}$ $(0\le j\in\mathbb{Z})$
such that $\sum_{j\ge0}a_{j}b_{jn}<\infty$ for all $n$. 
\end{lem}

\begin{proof}
This is a diagonal argument. First choose a decreasing sequence of
positive numbers $A_{1}^{1},A_{2}^{1},\ldots$ such that $\sum_{j\ge0}A_{j}^{1}b_{j1}<\infty$.
Leaving $A_{1}^{1}$ fixed \emph{shrink} $\{A_{j}^{1}\}_{j\ge2}$
if necessary to obtain $\{A_{j}^{2}\}_{j\ge2}$ so that $\sum_{j\ge0}A_{j}^{2}b_{j2}<\infty$.
Leaving $A_{1}^{1},A_{2}^{2}$ fixed \emph{shrink} $\{A_{j}^{2}\}_{j\ge2}$
if necessary to obtain $\{A_{j}^{3}\}_{j\ge3}$ so that $\sum_{j\ge0}A_{j}^{3}b_{j3}<\infty$.
Continuing similarly we obtain the desired sequence $a_{j}=A_{j}^{j}$. 
\end{proof}
Let $\Phi=\sum_{j\ge0}(-1)^{j}b_{j}\phi^{(2j)}$ then $\phi=\Phi*\psi$
where the last equality is in $\mathcal{S}'$. Since both sides are
Schwartz functions we have equality in $\mathcal{S}$. This proves
theorem \ref{thm: Schwartz}.
\begin{rem}
This strong factorization result is specific to Schwartz space on
the real line and analogous Schwartz spaces used in the theory of
automorphic forms need not have strong factorization. Strong factorization
provably fails for test functions on $\mathbb{R}^{n}$ for $n\ge3$
\cite{Dixmier_Malliavin_Failure}. 
\end{rem}

\subsection{\label{subsec:Some-technical-results.}Some technical results. }

Let $\mathcal{D}$ denote the space of test functions, smooth functions
on $\mathbb{R}$ with compact support. We now prove the test function
analogue of corollary \ref{cor:bounds b_n Schwartz} above about Schwartz
functions. Now we add a further hypothesis that the sequences satisfy
$a_{n}\ge n$. Then we have a uniform bound on their derivatives.
\begin{claim}
\label{claim: uniform bounds}We may find a sequence $c_{j}$ satisfying
$c_{j}>\sup_{x\in\mathbb{R}}|\phi^{(j)}(x)|$ independent of all $\phi$
constructed such that 

\[
\widehat{\phi}(\xi)=\frac{1}{\sum_{n\ge j\ge0}(2\pi)^{j}b_{j}\xi^{2j}}=\frac{1}{\prod_{n\ge1}\left(1+\frac{\xi^{2}}{a_{n}^{2}}\right)}
\]
with bounds $0<b_{n}<\frac{1}{(2\pi)^{j}(n!)^{4}}$ and $a_{n}\ge n$. 
\end{claim}

\begin{proof}
Since $a_{n}\ge n$ it follows that 
\[
\prod_{n\ge1}\left(1+\frac{\xi^{2}}{a_{n}^{2}}\right)\ge\prod_{m\le j+2}\left(1+\frac{\xi^{2}}{m^{2}}\right)
\]
Thus we have an estimate on $\widehat{\phi}$ independent of $\phi$
\[
|\widehat{\phi}(\xi)|\le\frac{1}{\prod_{1\le m\le j+2}\left(1+\frac{\xi^{2}}{m^{2}}\right)}
\]
 By Sobolev embedding theorem, 

\[
\sup_{x\in\mathbb{R}}|\phi^{(j)}(x)|\le C\int_{\mathbb{R}}(1+\xi{}^{2})^{j+1}\widehat{\phi}(\xi)d\xi
\]
for a constant $C$ independent of $\phi$. The result follows by
the estimate above on $\widehat{\phi}$. 
\end{proof}
\begin{lem}
\label{lem: bounds test}Given a sequence of bounds $B_{n}>0$, there
exists a sequence of constants $0<b_{n}<B_{n}$ and a function $f\in\mathcal{D}$
such that 
\[
\sum_{i=1}^{n}(-1)^{j}b_{j}f^{\left(2j\right)}\to\delta+h\hspace*{1em}(\text{in }\mathcal{D}')
\]
 as $n\to\infty$ for some $h\in\mathcal{D}$. The support of the
test functions $f$ and $h$ can be made arbitrarily small. 
\end{lem}

\begin{proof}
Let $\omega$ be a test function symmetric about the origin, supported
on $[-2,2]$ and identically $1$ on $[-1,1]$. Let $\phi_{\lambda}\in\mathcal{S}$
be as in the similar lemma above about Schwartz functions with bounds
$\lambda=\{\lambda_{j}\}_{j\ge0}$ to be decided later. The function
$f=\omega\cdot\phi_{\lambda}$ is supported on $[-2,2]$ and $\phi_{\lambda}=f$
on $[-1,1]$. By claim \ref{claim: uniform bounds} we can choose
$c_{n}=\sup_{x\in\mathbb{R}}|f^{(n)}(x)|$ independent of the sequence
$\lambda$. We pick the bounds $b_{j}$ for $\phi_{\lambda}=\phi$,
namely $b_{j}<\min\{\lambda_{j},\frac{1}{j^{2}c_{2j}},\ldots,\frac{1}{j^{2}c_{2j+j}}\}$
satisfying the the same assumptions as $\lambda$. Now 
\[
\sum_{j\le n}(-1)^{j}b_{j}f^{(2j)}\to\delta+h\hspace*{1em}\text{(as }n\to\infty\text{)}
\]
for some test function $h$ supported in $[-2,2]$. Fore $|x|\ge2$
the entire sequence is zero, for $|x|\le1$ then $f=\phi$. On $|x|\ge1$, 

\[
|b_{j}f^{(2n+j)}(x)|\le b_{j}c_{2n+j}\le\frac{1}{j^{2}}
\]
 for $j$ sufficiently large so that $\sum_{j\le m}(-1)^{j}b_{j}f^{(2j)}$ converges to a smooth
function on $|x|>1$ as $m\to\infty$. By scaling one may make the support of $f$
and $h$ arbitrarily small. 
\end{proof}
\begin{rem}
Using the method in the proof of theorem \ref{thm: Schwartz} any
test function $\phi$ on the real line can be written as $\phi=\Phi*f-h*\phi$.
Again, strong factorization provably fails for test functions on $\mathbb{R}^{n}$
for $n\ge3$ \cite{Dixmier_Malliavin_Failure}. 
\end{rem}

The Dixmier-Malliavin theorem for Lie groups is proved by a clever
use of the theorem in the case of the real line. This method requires
a result about convolution of certain compactly supported measures
on Lie groups, proven below. 

Let $\{x_{1},\ldots,x_{n}\}$ be a basis of $\mathfrak{g}$, the Lie
algebra of a real Lie group $G$, and let $\Theta:\mathfrak{g}\to G$
be the map given by $t_{1}x_{1}+\ldots+t_{n}x_{n}\mapsto\exp(t_{1}x_{1})\ldots\exp(t_{n}x_{n})$.
The map $\Theta$ is a local diffeomorphism at $0\in\mathfrak{g}.$
Let $V\subset G$ be an open neighborhood of the identity where $\Theta^{-1}$ exists and let
$\Theta^{-1}:V\to U\subset\mathfrak{g}\simeq\mathbb{R}^{n}$ be $\Theta^{-1}(g)=(\theta_{1}(g),\ldots,\theta_{n}(g))$.
For $x\in\mathfrak{g}$ and $f\in\mathcal{D}(\mathbb{R})$ we define
a measure $\mu_{x,f}$ on $G$ by invoking the Riesz-Markov-Kakutani
theorem as follows. 
\[
\mu_{x,f}(F)=\int_{\mathbb{R}}F(e^{tx})f(t)dt\hspace*{1em}(\text{for all }F\in C_{c}^{0}(G))
\]
 For any test function $\nu$ on $G$ and $y\in G$ let $R_{y}v$
denote the right regular action of $y$ on $\nu$. 
\begin{lem}
\label{lem:convolution of distribution}Let $f_{1},\ldots,f_{n}\in\mathcal{D}(\mathbb{R})$
be such that the support of $\mu_{x_{i},f_{i}}$ is contained in an
open neighborhood $W$ of $1\in G$ satisfying $W^{n}=W\cdots W\subset V$.
Then the distribution $\mu_{x_{1},f_{1}}*\cdots*\mu_{x_{1},f_{n}}$
is integration against a test function on $G$. 
\end{lem}

\begin{proof}
We denote $\mu_{x_{i},f_{i}}$ by $\mu_{i}$ for clarity. The convolution of measures $\mu_{x_{1},f_{1}}*\cdots*\mu_{x_{1},f_{n}}$
as a distribution is given by (see the remark below)
\[
(\mu_{x_{1},f_{1}}*\cdots*\mu_{x_{1},f_{n}})(\phi)=\int_{G^{n}}\phi(y_{1}\ldots y_{n})d\mu_{1}(y_{1})\ldots d\mu_{n}(y_{n})
\]

\[
=\int_{G^{n-1}}\left(\int_{G}(R_{y_{2}}\circ\cdots\circ R_{y_{n}}\phi)(y_{1})d\mu_{1}(y_{1})\right)d\mu_{2}\ldots d\mu_{n}
\]
\[
=\int_{G^{n-1}}\left(\int_{\mathbb{R}}(R_{y_{2}}\circ\cdots\circ R_{y_{n}}\phi)(e^{t_{1}x_{1}})f_{1}(t_{1})dt_{1}\right)d\mu_{2}\ldots d\mu_{n}
\]
\[
=\int_{G^{n-2}}\left(\int_{\mathbb{R}^{2}}(R_{y_{3}}\circ\cdots\circ R_{y_{n}}\phi)(e^{t_{1}x_{1}}e^{t_{2}x_{2}})\cdot f_{1}(t_{1})f_{1}(t_{1})dt_{1}dt_{2}\right)d\mu_{3}\ldots d\mu_{n}
\]

\[
\cdots=\int_{\mathbb{R}^{n}}\phi(e^{t_{1}x_{1}}\ldots e^{t_{n}x_{n}})f_{1}(t_{1})\ldots f_{n}(t_{n})dt_{1}\ldots dt_{n}
\]
\[
=\int_{V\subset G}\phi(g)(f\circ\theta_{1})(g)\ldots(f\circ\theta_{n})(g)\cdot J(g)dg
\]
\[
=\int_{G}\phi(g)(f\circ\theta_{1})(g)\ldots(f\circ\theta_{n})(g)\cdot J(g)dg
\]
where $dg$ is the Haar measure on $G$ and $J$ is the Jacobian arising
due to change of measure. Thus, the distribution $\mu_{x_{1},f_{1}}*\cdots*\mu_{x_{n},f_{n}}$
is given by $f\circ\theta_{1}(g)\cdots f\circ\theta_{n}(g)\cdot J(g)$
which is a test function on $G$. 
\end{proof}
\begin{rem}
The expression used above for the convolution can be seen for convolution
of two test functions $f,g$ by looking at the integral 
\[
\left(f*g\right)(\phi)=\int_{G}\int_{G}f(xy^{-1})g(y)\phi(x)dydx
\]
 
\[
=\int_{G}\int_{G}f(x)g(y)\phi(xy)dxdy\hspace{1em}\text{(Fubini and }x\to xy\text{)}
\]
 and for $n$ functions induction gives the desired expression. For
general compactly supported distributions $\lambda_{1},\ldots,\lambda_{n}$
the assertion is $\left(\lambda_{1}*\cdots*\lambda_{n}\right)(\phi)=(\lambda_{1}\otimes\cdots\otimes\lambda_{n})\left(\phi\circ M^{n}\right)$
where $M^{n}:G^{n}\to G$ given by $(x_{1},\ldots,x_{n})\mapsto x_{1}\cdots x_{n}$,
the tensor product is in the sense of the Schwartz kernel theorem,
and $\lambda_{1}\otimes\cdots\otimes\lambda_{n}$ is a distribution
on $G^{n}$. For measures this will give the expression used above. 
\end{rem}

\subsection{Proof of the Dixmier-Malliavin theorems}

Let $\mathcal{D}(G)$ be the space of smooth functions with compact
support on a real Lie group $G$. 
\begin{thm}
(Dixmier-Malliavin) Every element $\phi\in\mathcal{D}(G)$ is a finite
sum of convolutions $\phi_{i}*\psi_{i}$ where $\phi_{i},\psi_{i}\in\mathcal{D}(G)$
\end{thm}

\begin{proof}
We begin with an observation. Let $\{\gamma_{1},\ldots,\gamma_{m}\}$
be a basis for the Lie algebra $\mathfrak{g}$ of $G$ such that 
\[
(t_{1},\ldots,t_{m})\to\exp(t_{1}\gamma_{1})\ldots\exp(t_{m}\gamma_{m})
\]
 is a diffeomorphism of $(-1,1)^{m}$ to a neighborhood of $1\in G$.
Let $\{\alpha_{j}\}$ be a basis of the universal enveloping algebra
$U\mathfrak{g}$, and put 
\[
M_{\ell n}=\sup_{g\in G}|\alpha_{\ell}x_{1}^{2n}\phi(g)|\text{ \ \ \  (where \ensuremath{\alpha_{\ell}x_{1}^{2n}\in U\mathfrak{g}})}
\]
 By a diagonal argument above we can choose $b_{j}$ positive such
that $\sum_{j}b_{j}M_{j\ell}<+\infty$ for all $\ell$. This will
show that various summations below will converge in the space of test
functions. 

Given $g\in\mathcal{D}(\mathbb{R})$ and $x\in\mathfrak{g}$ define
a measure $\mu=\mu(x,g)$ on $G$ by 
\[
\int_{G}Fd\mu=\int_{\mathbb{R}}F(e^{tx})g(t)dt\hspace{1em}\hspace{1em}\hspace{1em}\text{ for }F\in C_{c}^{0}(G)
\]
 By lemma \ref{lem: bounds test} with bounds $b_{j}$ we can find
$f,h\in\mathcal{D}$ such that 
\[
\sum_{0\le j\le n}(-1)^{j}b_{j}f^{(2j)}\to\delta+h
\]
 as $n\to\infty$ in $\mathcal{D}'$. For $x\in\mathfrak{g}$,
\[
\left(\sum_{0\le j\le n}(-1)^{j}b_{j}x^{(2j)}\right)\cdot\mu(x,f)\to\delta+\mu(x,h)
\]
Given $\varphi\in\mathcal{D}(G)$, on one hand
\[
\left(\left(\sum(-1)^{j}b_{j}x^{(2j)}\right)\cdot\mu(x,f)\right)*\varphi\to\varphi+\mu(x,h)*\varphi
\]
On the other hand,

\[
\left(\left(\sum(-1)^{j}b_{j}x^{(2j)}\right)\cdot\mu(x,f)\right)*\varphi\to\mu(x,f)*\Phi
\]
where $\Phi=\sum_{j\ge0}(-1)^{j}b_{j}\varphi^{(2j)}$. Thus, 
\[
\varphi=-\mu(x,h)*\varphi+\mu(x,f)*\Phi
\]
If $\{x_{1},\ldots,x_{n}\}$ is a basis for $\mathfrak{g}$ then by
applying the above to $\varphi$ and $\Phi$ on the right $(n-1)$-times
we can write $\varphi$ as a linear combination of terms of the form
$\mu(x_{1},f_{1})*\ldots*\mu(x_{n},f_{n})*\alpha$ where $f_{i}\in\mathcal{D}(\mathbb{R})$,
$\alpha\in\mathcal{D}(G)$ and $\mu(x_{i},f_{i})$ are compactly supported
measures with support contained in $\{\exp(tx_{i}):t\in\mathbb{R}\}$.
We may further make sure that support of $f_{i}$ and hence $\mu(x_{i},f_{i})$
are arbitrarily small so that computations can be done in local coordinates.
By lemma \ref{lem:convolution of distribution} above about convolution
of certain compactly supported measures, $\mu(x_{1},f_{1})*\ldots*\mu(x_{n},f_{n})$
is a test function. 
\end{proof}
\begin{thm}
(Dixmier-Malliavin theorem for representations) Let $(\pi,V)$ be
a smooth Fr\'echet space representation of a real Lie group $G$ and
let $V^{\infty}$ be the Fr\'echet space of smooth vectors in $V$.
Then every smooth vector can be written as a finite linear combination
of vectors of the form $\pi(f)v$ where $f\in\mathcal{D}(G)$ and
$v\in V^{\infty}$. 
\end{thm}

\begin{proof}
Let $|.|_{\ell}$ be a countable semi-norms on $V$ corresponding
to the topology on $V$. Let $\{x_{j}\}$ be a basis for $U\mathfrak{g}$,
the universal enveloping algebra of $\mathfrak{g}$ and set
\[
M_{n,k,\ell}=|\pi(X_{k}x^{2n})v|_{\ell}
\]
 By a diagonal argument similar to the one above we may choose $b_{n}$
such that $\sum_{n}b_{n}M_{n,k,\ell}<\infty$ for all $k,\ell$. By
lemma \ref{lem: bounds test} with bounds $b_{j}$ we can find $f,h\in\mathcal{D}$
such that $\sum_{0\le j\le n}(-1)^{j}b_{j}f^{(2j)}\to\delta+h$ as
$n\to\infty$ in $\mathcal{D}'$. We have
\[
\pi(\mu_{x,f})*\left(\sum_{0\le j\le n}(-1)^{j}b_{n}\pi(x^{2n})\right)\cdot v\to v+\pi(\mu_{x,h})v
\]
 in $V$. By the above considerations $\sum_{0\le j\le n}(-1)^{j}b_{n}\pi(x^{2n})v$
has a limit $\eta$ in $V$ and then $\pi(\mu_{x,f})\eta=v+\pi(\mu_{x,h})v$.
As in the previous theorem $v$ is a linear combination of vectors
of the form $\pi(\mu_{x_{1},f_{1}}*\cdots*\mu_{x_{n},f_{n}})w$ for
some $w\in V$. Again by using lemma \ref{lem:convolution of distribution}
about convolution of compactly supported measures the theorem follows. 
\end{proof}
\begin{rem}
The space of test functions on a non-compact manifold with the usual
colimit topology are \emph{never} Fr\'echet spaces since they are a
countable union of closed sets with empty interior and hence would
violate the conclusions of the Baire category theorem. Thus, Dixmier-Malliavin
theorem for representations does not directly imply the result for
test functions. Smooth vectors in non-Fr\'echet space representations
are more subtle. For example, for the regular representation of $\mathbb{R}$
on $\mathcal{D}^{*}$ \emph{every} vector is a smooth vector since
distributions are infinitely differentiable. Setting up a common notation
to treat both the cases in a single theorem does not appear to be
worth the clarity lost in the process. 
\end{rem}

\medskip

\printbibliography
\end{document}